\documentclass[12pt]{amsart}
\usepackage{amsmath,amssymb,amsbsy,amsfonts,amsthm,latexsym,
            amsopn,amstext,amsxtra,euscript,amscd}
            \usepackage{url}
\usepackage[colorlinks,linkcolor=blue,anchorcolor=blue,citecolor=blue]{hyperref}
\usepackage{color}

\begin{document}
\bibliographystyle{plain}
\newfont{\teneufm}{eufm10}
\newfont{\seveneufm}{eufm7}
\newfont{\fiveeufm}{eufm5}
%
%
\newfam\eufmfam
              \textfont\eufmfam=\teneufm \scriptfont\eufmfam=\seveneufm
              \scriptscriptfont\eufmfam=\fiveeufm
\def\bbbr{{\rm I\!R}}
\def\bbbm{{\rm I\!M}}
\def\bbbn{{\rm I\!N}}
\def\bbbf{{\rm I\!F}}
\def\bbbh{{\rm I\!H}}
\def\bbbk{{\rm I\!K}}
\def\bbbp{{\rm I\!P}}
\def\bbbone{{\mathchoice {\rm 1\mskip-4mu l} {\rm 1\mskip-4mu l}
{\rm 1\mskip-4.5mu l} {\rm 1\mskip-5mu l}}}
\def\bbbc{{\mathchoice {\setbox0=\hbox{$\displaystyle\rm C$}\hbox{\hbox
to0pt{\kern0.4\wd0\vrule height0.9\ht0\hss}\box0}}
{\setbox0=\hbox{$\textstyle\rm C$}\hbox{\hbox
to0pt{\kern0.4\wd0\vrule height0.9\ht0\hss}\box0}}
{\setbox0=\hbox{$\scriptstyle\rm C$}\hbox{\hbox
to0pt{\kern0.4\wd0\vrule height0.9\ht0\hss}\box0}}
{\setbox0=\hbox{$\scriptscriptstyle\rm C$}\hbox{\hbox
to0pt{\kern0.4\wd0\vrule height0.9\ht0\hss}\box0}}}}
\def\bbbq{{\mathchoice {\setbox0=\hbox{$\displaystyle\rm
Q$}\hbox{\raise
0.15\ht0\hbox to0pt{\kern0.4\wd0\vrule height0.8\ht0\hss}\box0}}
{\setbox0=\hbox{$\textstyle\rm Q$}\hbox{\raise
0.15\ht0\hbox to0pt{\kern0.4\wd0\vrule height0.8\ht0\hss}\box0}}
{\setbox0=\hbox{$\scriptstyle\rm Q$}\hbox{\raise
0.15\ht0\hbox to0pt{\kern0.4\wd0\vrule height0.7\ht0\hss}\box0}}
{\setbox0=\hbox{$\scriptscriptstyle\rm Q$}\hbox{\raise
0.15\ht0\hbox to0pt{\kern0.4\wd0\vrule height0.7\ht0\hss}\box0}}}}
\def\bbbt{{\mathchoice {\setbox0=\hbox{$\displaystyle\rm
T$}\hbox{\hbox to0pt{\kern0.3\wd0\vrule height0.9\ht0\hss}\box0}}
{\setbox0=\hbox{$\textstyle\rm T$}\hbox{\hbox
to0pt{\kern0.3\wd0\vrule height0.9\ht0\hss}\box0}}
{\setbox0=\hbox{$\scriptstyle\rm T$}\hbox{\hbox
to0pt{\kern0.3\wd0\vrule height0.9\ht0\hss}\box0}}
{\setbox0=\hbox{$\scriptscriptstyle\rm T$}\hbox{\hbox
to0pt{\kern0.3\wd0\vrule height0.9\ht0\hss}\box0}}}}
\def\bbbs{{\mathchoice
{\setbox0=\hbox{$\displaystyle     \rm S$}\hbox{\raise0.5\ht0\hbox
to0pt{\kern0.35\wd0\vrule height0.45\ht0\hss}\hbox
to0pt{\kern0.55\wd0\vrule height0.5\ht0\hss}\box0}}
{\setbox0=\hbox{$\textstyle        \rm S$}\hbox{\raise0.5\ht0\hbox
to0pt{\kern0.35\wd0\vrule height0.45\ht0\hss}\hbox
to0pt{\kern0.55\wd0\vrule height0.5\ht0\hss}\box0}}
{\setbox0=\hbox{$\scriptstyle      \rm S$}\hbox{\raise0.5\ht0\hbox
to0pt{\kern0.35\wd0\vrule height0.45\ht0\hss}\raise0.05\ht0\hbox
to0pt{\kern0.5\wd0\vrule height0.45\ht0\hss}\box0}}
{\setbox0=\hbox{$\scriptscriptstyle\rm S$}\hbox{\raise0.5\ht0\hbox
to0pt{\kern0.4\wd0\vrule height0.45\ht0\hss}\raise0.05\ht0\hbox
to0pt{\kern0.55\wd0\vrule height0.45\ht0\hss}\box0}}}}
\def\bbbz{{\mathchoice {\hbox{$\sf\textstyle Z\kern-0.4em Z$}}
{\hbox{$\sf\textstyle Z\kern-0.4em Z$}}
{\hbox{$\sf\scriptstyle Z\kern-0.3em Z$}}
{\hbox{$\sf\scriptscriptstyle Z\kern-0.2em Z$}}}}
\def\ts{\thinspace}

\newtheorem{theorem}{Theorem}
\newtheorem{corollary}[theorem]{Corollary}
\newtheorem{lemma}[theorem]{Lemma}
\newtheorem{claim}[theorem]{Claim}
\newtheorem{cor}[theorem]{Corollary}
\newtheorem{prop}[theorem]{Proposition}
\newtheorem{definition}[theorem]{Definition}
\newtheorem{remark}[theorem]{Remark}
\newtheorem{question}[theorem]{Open Question}

\numberwithin{equation}{section}
\numberwithin{theorem}{section}

\def\qed{\ifmmode
\squareforqed\else{\unskip\nobreak\hfil
\penalty50\hskip1em\null\nobreak\hfil\squareforqed
\parfillskip=0pt\finalhyphendemerits=0\endgraf}\fi}

\def\squareforqed{\hbox{\rlap{$\sqcap$}$\sqcup$}}

\def \A {{\mathbb A}}
\def \C {{\mathbb C}}
\def \F {{\mathbb F}}
\def \L {{\mathbb L}}
\def \K {{\mathbb K}}
\def \Q {{\mathbb Q}}
\def \Z {{\mathbb Z}}
\def\cA{{\mathcal A}}
\def\cB{{\mathcal B}}
\def\cC{{\mathcal C}}
\def\cD{{\mathcal D}}
\def\cE{{\mathcal E}}
\def\cF{{\mathcal F}}
\def\cG{{\mathcal G}}
\def\cH{{\mathcal H}}
\def\cI{{\mathcal I}}
\def\cJ{{\mathcal J}}
\def\cK{{\mathcal K}}
\def\cL{{\mathcal L}}
\def\cM{{\mathcal M}}
\def\cN{{\mathcal N}}
\def\cO{{\mathcal O}}
\def\cP{{\mathcal P}}
\def\cQ{{\mathcal Q}}
\def\cR{{\mathcal R}}
\def\cS{{\mathcal S}}
\def\cT{{\mathcal T}}
\def\cU{{\mathcal U}}
\def\cV{{\mathcal V}}
\def\cW{{\mathcal W}}
\def\cX{{\mathcal X}}
\def\cY{{\mathcal Y}}
\def\cZ{{\mathcal Z}}
\newcommand{\rmod}[1]{\: \mbox{mod}\: #1}

\def\tcN{\cN^\mathbf{c}}
\def\F{\mathbb F}
\def\Tr{\operatorname{Tr}}
\def\mand{\qquad \mbox{and} \qquad}
\renewcommand{\vec}[1]{\mathbf{#1}}
\def\eqref#1{(\ref{#1})}
\newcommand{\ignore}[1]{}
\hyphenation{re-pub-lished}
\parskip 1.5 mm
\def\lln{{\mathrm Lnln}}
\def\Res{\mathrm{Res}\,}
\def\lcm{\mathrm{lcm}\,}
\def\rad{\mathrm{rad}\,}
\def\F{{\bbbf}}
\def\Fp{\F_p}
\def\fp{\Fp^*}
\def\Fq{\F_q}
\def\ff{\F_2}
\def\ffn{\F_{2^n}}
\def\K{{\bbbk}}
\def \Z{{\bbbz}}
\def \N{{\bbbn}}
\def\Q{{\bbbq}}
\def \R{{\bbbr}}
\def \P{{\bbbp}}
\def\Zm{\Z_m}
\def \Um{{\mathcal U}_m}
\def \Bf{\frak B}
\def\Km{\cK_\mu}
\def\va {{\mathbf a}}
\def \vb {{\mathbf b}}
\def \vc {{\mathbf c}}
\def\vx{{\mathbf x}}
\def \vr {{\mathbf r}}
\def \vv {{\mathbf v}}
\def\vu{{\mathbf u}}
\def \vw{{\mathbf w}}
\def \vz {{\mathbfz}}
\def\\{\cr}
\def\({\left(}
\def\){\right)}
\def\fl#1{\left\lfloor#1\right\rfloor}
\def\rf#1{\left\lceil#1\right\rceil}
\def\AST{\mathcal{A}_{\mathrm{ST}}}
\def\AU{\mathcal{A}_{\mathrm{U}}}

\def \ctE {\widetilde \cE}

\newcommand{\floor}[1]{\lfloor {#1} \rfloor}
\newcommand{\commF}[1]{\marginpar{%
\begin{color}{green}
\vskip-\baselineskip 
\raggedright\footnotesize
\itshape\hrule \smallskip F: #1\par\smallskip\hrule\end{color}}}

\newcommand{\commI}[1]{\marginpar{%
\begin{color}{blue}
\vskip-\baselineskip 
\raggedright\footnotesize
\itshape\hrule \smallskip I: #1\par\smallskip\hrule\end{color}}}

\newcommand{\commM}[1]{\marginpar{%
\begin{color}{magenta}
\vskip-\baselineskip 
\raggedright\footnotesize
\itshape\hrule \smallskip M: #1\par\smallskip\hrule\end{color}}}

\def\rem{{\mathrm{\,rem\,}}}
\def\dist {{\mathrm{\,dist\,}}}
\def\etal{{\it et al.}}
\def\ie{{\it i.e. }}
\def\veps{{\varepsilon}}
\def\eps{{\eta}}
\def\ind#1{{\mathrm {ind}}\,#1}
               \def \MSB{{\mathrm{MSB}}}
\newcommand{\abs}[1]{\left| #1 \right|}

\title[Functions arising from Euler and Carmichael quotients]{On two functions arising in the study of the Euler and Carmichael quotients}

\author{Florian~Luca}
\address{School of Mathematics, University of the Witwatersrand, Private Bag X3, Wits 2050, South Africa, 
and Department of Mathematics, Faculty of Sciences, University of Ostrava, 30. dubna 22, 701 03 Ostrava 1, Czech Republic}
\email{florian.luca@wits.ac.za}

\author{Min Sha}
\address{School of Mathematics and Statistics, University of New South Wales,
 Sydney, NSW 2052, Australia}
\email{shamin2010@gmail.com}

\author{Igor E. Shparlinski}
\address{School of Mathematics and Statistics, University of New South Wales,
 Sydney, NSW 2052, Australia}
\email{igor.shparlinski@unsw.edu.au}

\date{}

\subjclass[2010]{11A25, 11K65,  11N36}
\keywords{Euler quotient, Carmichael quotient, Carmichael function, Euler function, greatest common divisor}

\begin{abstract}
We investigate two arithmetic functions naturally occurring in the study of the Euler and Carmichael quotients. 
The functions are related to the frequency of vanishing of the Euler and Carmichael quotients. We obtain 
several results concerning  the relations between these functions as well as their typical and extreme values.
\end{abstract}

\maketitle

\section{Introduction}
For  a positive integer $m$, let $\lambda(m)$ be the exponent of the multiplicative group modulo $m$, which is the so-called Carmichael function of $m$, and let $\varphi(m)$ be the Euler function of $m$. If the prime factorization of $m$ is
\begin{equation}
\label{eq:1}
m=p_1^{r_1}\ldots p_k^{r_k},
\end{equation}
then 
$$
\lambda(m)=\lcm[\lambda(p_1^{r_1}), \lambda(p_2^{r_2}),\ldots, \lambda(p_k^{r_k})],
$$
where  for a prime power $p^r$ we have $\lambda(p^r)=p^{r-\sigma}(p-1)$ with $\sigma=1$ except the case when $p=2$ and $r\ge 3$  in which $\sigma=2$. 

Given an integer $a$ relatively prime to $m$, the integer 
$$
Q_m(a) = \frac{a^{\varphi(m)}-1}{m}
$$
is called the \textit{Euler quotient} of $m$ with base $a$, 
which is a generalization of the classical \textit{Fermat quotient}; besides, the integer 
$$
C_m(a)=\frac{a^{\lambda(m)}-1}{m}
$$
is called the \textit{Carmichael quotient} of $m$ with base $a$. 
 Some arithmetic properties of the numbers $Q_m(a)$ and $C_m(a)$ appear in~\cite{ADS} and~\cite{Sha}, respectively. For example, if  the integers $a,~b$ are coprime to $m$, it is shown 
 in~\cite[Proposition 2.1]{ADS} and~\cite[ Proposition~2.2]{Sha}, respectively, that 
$$
Q_m(ab)\equiv Q_m(a)+Q_m(b)\pmod m
$$
and 
$$
C_m(ab)\equiv C_m(a)+C_m(b)\pmod m; 
$$ 
 if furthermore 
$a\equiv b\pmod{m^2}$, then $Q_m(a)\equiv Q_m(b)\pmod m$ and $C_m(a)\equiv C_m(b)\pmod m$.

In particular, we can define two group homomorphisms:
$$
\psi_m:\  (\Z/m^2\Z)^*\longrightarrow (\Z/m\Z,+),\qquad x\mapsto Q_m(x),
$$
and 
$$
\phi_m:\  (\Z/m^2\Z)^*\longrightarrow (\Z/m\Z,+),\qquad x\mapsto C_m(x).
$$

 For a positive integer $m$ with prime factorization~\eqref{eq:1}, recall that the radical of $m$ is defined to be the product of all distinct prime factors of $m$:
$$
\rad(m)=p_1\ldots p_k.
$$
By~\cite[Proposition~4.4]{ADS}, the image of the morphism $\psi_m$ is $d(m)\Z/m\Z$, where 
$$
d(m)=\gcd(m, \delta \varphi(\rad(m)) ), 
$$
where $\delta=2$ if $4\mid m$ and $\delta=1$ otherwise. 
The image of the morphism $\phi_m$ can also be  determined. The related result in~\cite[Proposition~4.3]{Sha}  concerning the image of $\phi_m$ is not correct and should be replaced by the statement that the image of $\phi_m$ is $f(m)\Z/m\Z$, where 
$$
f(m) = \gcd(m, \delta \lambda(m) \varphi(\rad(m))/\varphi(m))
$$
(see Proposition 4.3 in the arXiv version of~\cite{Sha}). So, $f(m) \mid d(m)$. 
Moreover, by~\cite[Proposition~2.1]{Sha}, we get 
$$
\frac{\varphi(m)}{\lambda(m)}f(m)\Z/m\Z = d(m)\Z/m\Z, 
$$
which implies that $\gcd(m,\frac{\varphi(m)}{\lambda(m)}f(m))=d(m)$. Clearly, we have 
$$
\frac{d(m)}{f(m)} = \gcd(m/f(m), \varphi(m)/\lambda(m)). 
$$
Furthermore, by~\cite[Corollary~4.3]{Sha}, we have 
\begin{equation}
\label{eq:Zero Set}
\#\{a \in(\Z/m^2\Z)^*~:~C_m(a) \equiv 0 \pmod m\} = f(m)\varphi(m).
\end{equation}

Some questions about $d(m)$ and $f(m)$ have been studied in~\cite{Sha}. For example, in the proof
of~\cite[Proposition~4.5]{Sha} it has been  shown that $\lim_{m\to \infty} d(m)/m=0$. Further, it is shown
in~\cite[Section~4]{Sha}  that under the assumption of the existence
of infinitely many Sophie-Germain primes, we have
$$
\limsup _{m\to\infty} d(m)/m^{1/2} \ge 1/\sqrt{2}.
$$
Here, using a result and some arguments from~\cite{ELP},  we make these inequalities more precise
as follows:

\begin{theorem}  
\label{thm:dm}
 We have:
\begin{itemize}
\item[(i)] For any integer $m\ge 1$,
$$
d(m)\le \sqrt{2} m \exp\left(-{\sqrt{\log 2 \log m + (\log 2)^2/4}}\right).
$$
\item[(ii)] For infinitely many  integers $m\ge 1$,
$$
f(m)  \ge  m^{1-(1+o(1))\log\log\log m/\log\log m}.
$$
\end{itemize}
\end{theorem}

When $m$ is square-free, it is easy to see that $d(m)=f(m)$. 
Our next result is to show that this is almost always true.

\begin{theorem}
\label{thm:fm=dm}
The set of positive integers $m$ such that $d(m)=f(m)$ is of asymptotic density $1$. 
\end{theorem}

More precisely, according to the proof of Theorem~\ref{thm:fm=dm},  for sufficiently large $x$ and for all positive integers $m\le x$ outside a subset of cardinality $o(x)$, we have 
\begin{equation} \label{eq:dfm}
d(m)=f(m)=\prod_{\substack{1\le j \le k\\ p_j< y/\log y}} p_j^{r_j}. 
\end{equation}
 where $y=\log\log x$, and we have assumed that $m$ has the prime factorization~\eqref{eq:1}. Furthermore, we can replace $ y/\log y$ by $y$ in~\eqref{eq:dfm}. Indeed, it is easy to see that the set of positive integers not exceeding $x$ and divisible by a prime in the interval $[ y/\log y,y)$ has asymptotic density 0 by considering the reciprocal sum of the primes $p$ in the interval and using the Mertens formula 
 \begin{equation}\label{eq:Mert}
\sum_{\textrm{prime $p\le t$}} \frac{1}{p}  = \log \log t + A + O(1/\log t), \qquad t \ge 3,
\end{equation}
with some constant $A$, see~\cite[Equation~(2.15)]{IwKow}. 

Finally, since $f(m)\mid d(m)$ and $f(m)=d(m)$ for almost all $m$, it makes sense to ask which values are possible for $f(m)$ and $d(m)$.  
One may conjecture that 
 for each fixed positive integers $a$ and $b$ with $a \mid b$ there exists $m$ such
that $(f(m),d(m))=(a,b)$. We now establish this for  large families of pairs $(a,b)$
but also show that this conjecture is false in general. 

  Before we formulate our next result we need to recall that  the notations  $U \ll V$ and $U = O(V)$, are equivalent to $|U|  \le c V$ for some constant $c>0$. 
As usual, $U = o(V)$ means that $U/V \to 0$, and $U \sim V$ means that $U/V \to 1$.

Recall that \textit{Linnik's Theorem} asserts that there exists a positive number $L$, 
known as \textit{Linnik's constant}, such that,
if $p(a,d)$ denotes the smallest prime in the arithmetic progression $\{a+nd: n\ge 0\}$ for integers 
$1\le a < d$ with $\gcd(a,d)=1$, then $p(a,d) \ll d^L$. 
It is known~\cite{BFI} that $L \le 2$ for almost all integers $d$. 
Currently, the best general estimate is $L \le 5$, due to Xylouris~\cite{Xyl2} (see also~\cite{Xyl1} 
for $L = 5.18$, which improves the 
previous bound $L \le 5.5$ of Heath-Brown~\cite{H-B1}); see also Section~\ref{sec:comm}
below for further comments on the choice of $L$.

\begin{theorem}
\label{thm: pair df} 
We have:
\begin{itemize}
\item[(i)] Given any positive integer $n$, there exists $m\ll n^{2L+1}$ such that $d(m)=f(m)=n$. 

\item[(ii)] Given any positive integer $n$, there exists $m\ll n^{6L+3}$ such that $d(m)/f(m)=n$. 

\item[(iii)] Let $a,b$ be two positive integers such that  $a\mid b$. Assume that  $\gcd(b/a,a\varphi(\rad(b)))=1$. Then, there exists $m\ll b^{2L+2}/a$ such that $(f(m),d(m))=(a,b)$.
\end{itemize}
\end{theorem}

We remark that the co-primality assumption in Theorem \ref{thm: pair df} (iii) might be strong. 
Because Erd\H os \cite{Erdos} has shown that the set of positive integers $n$ with $\gcd(n,\varphi(n))=1$ is of asymptotic density 0. 
However, when the assumption in Theorem~\ref{thm: pair df} (iii) does not hold, the situation becomes unstable. We give some examples as follows.

\begin{theorem} 
\label{thm:except}
We have:
\begin{itemize}
\item[(i)] Let $n$ be an odd positive integer greater than $1$. Then,  there does not exist an integer $m\ge 1$ such that $(f(m),d(m))=(n,2n)$ or $(n,4n)$.
\item[(ii)] Let $p,q$ be two odd primes such that $p<q, p\mid q-1$ and $p^2 \nmid q-1$. Then, there exists an integer $m \ll p(pq)^{L+1}$ such that $(f(m),d(m))=(q,pq)$.
\end{itemize}
\end{theorem}

\begin{remark}
{\rm
From the proofs in Sections \ref{sec:pair df} and \ref{sec:except}, one can see that for each result in Theorem \ref{thm: pair df} 
and in Theorem \ref{thm:except} (ii), there are infinitely many such integers $m$ if we drop the boundedness condition. 
}
\end{remark}

\section{Proof of Theorem~\ref{thm:dm}}
\label{sec:dm}

The closely related function
$$
D(m)=\gcd(m, \varphi(m))
$$
for square-free integers $m \ge 1$ has been studied in~\cite{ELP}. For example, 
it is shown in~\cite[Theorem~5.1]{ELP} that
for all square-free integers $m\ge 1$ we have
\begin{equation}
\label{eq:D Up}
D(m)\le 2m \exp(-{\sqrt{\log 2 \log m}}), 
\end{equation}
and for infinitely many square-free integers $m\ge 1$ we have
\begin{equation}
\label{eq:D Low}
D(m)\ge m^{1-(1+o(1))\log\log\log m/\log\log m}.
\end{equation} 

Since $d(m)=D(m)$ for square-free $m$, by~\eqref{eq:D Low} there are infinitely many integers $m\ge 1$ such that 
$$
d(m)\ge m^{1-(1+o(1))\log\log\log m/\log\log m}.
$$
Here, we want to establish a similar result for $f(m)$, as well as a non-trivial upper bound for $d(m)$. 

We start with an observation that a  modification
of the argument in the proof of~\cite[Theorem~5.1]{ELP}
allows us to obtain the following improvement upon~\eqref{eq:D Up}.

\begin{lemma} 
\label{lem:Dm}
For any square-free integer $m > 1$ and $m \ne 6$, we have
$$
D(m) \le  m \exp\(-\sqrt{\log 2 \log m}\).
$$
\end{lemma}

\begin{proof} 
First, the case $k=1$ can be checked directly, 
and thus we can assume that $k\ge 2$. 
Now, suppose that $m$ has the prime factorization~\eqref{eq:1}  such that 
$$
p_1< \cdots < p_k.
$$

In the proof of~\cite[Theorem~5.1]{ELP} it has been showed that the desired result is true when one of the following conditions holds:
\begin{enumerate}
\item $p_k > \exp\(\sqrt{\log 2 \log m}\)$,

\item  $m$ is odd.

\end{enumerate}
So, to complete the proof, in the sequel we assume that $m$ is even, that is,
\begin{equation}
\label{eq:p1 pk}
p_1=2 \mand p_k \le \exp\(\sqrt{\log 2 \log m}\).
\end{equation}

If $k=2$, then we must have $p_2 > 3$ since $m\ne 6$. In this case, we have $m=2p_2$ and $D(m)=2$. 
So, the result can also be checked directly. 

Now, we assume that $k \ge 3$, and then $m\ge 30$. 
Then, we deduce that 
\begin{equation*}
\begin{split}
D(m) & =  \gcd(2p_2\ldots p_k, (p_2-1)\ldots (p_k-1)) \\
& =  2 \gcd\(p_2p_3\ldots p_k, \frac{p_2-1}{2}\ldots \frac{p_k-1}{2}\) \\
& =  2 \gcd\(p_2p_3\ldots p_k, \frac{p_3-1}{2}\ldots \frac{p_k-1}{2}\) \\
& \le  2\left(\frac{p_3-1}{2}\right)\ldots \left(\frac{p_k-1}{2}\right) \\
& =   \frac{2 \varphi(m)}{(p_2-1)2^{k-2}} \le \frac{m}{2^{k-2}p_2},
\end{split}
\end{equation*}
since $m$ is even, see~\eqref{eq:p1 pk}. 
Besides, also by~\eqref{eq:p1 pk}, we have
$$
k - 2 \ge \frac{\log (m/(2p_2))}{\log p_k} \ge \frac{\log (m/(2p_2))}{\sqrt{\log 2 \log m}}.
$$
So
$$
D(m) \le  \frac{m}{p_2} \exp\(-\sqrt{\log 2 \log m} + \sqrt{\log 2 / \log m}\log(2p_2)\). 
$$
Thus, the result follows if 
\begin{equation} \label{eq:log2m}
\exp\(\sqrt{\log 2 / \log m}\log(2p_2)\) \le p_2. 
\end{equation} 
Since $m\ge 30$, the inequality~\eqref{eq:log2m} is implied in the following inequality 
$$
\sqrt{2p_2} \le p_2, 
$$
which is definitely true since $p_2 \ge 3$, and we conclude the proof. 
\end{proof}

We are now ready to prove  Theorem~\ref{thm:dm}.

\begin{proof}[Proof of Theorem~\ref{thm:dm}]
(i) We write $m=n r$, where $r=\rad(m)$.
If $r=6$, then $m=2^{s_1}3^{s_2}$ with $s_1\ge 1$ and $s_2\ge 1$. It is easy to check this case by a direct computation. In the following, we assume that $r\ne 6$.

On one hand,
\begin{equation}
\label{eq:2}
d(m)=\gcd(nr, \delta\varphi(r))\le \delta \varphi(r)\le r,
\end{equation}
where we have used the fact that $\delta \varphi(r)\le r$ (this is obvious for $\delta=1$, and when $\delta=2$, then $r$ is even, so $\varphi(r)\le r/2=r/\delta$). On the other hand,
$$
d(m)  =  \gcd(nr,\delta\varphi(r))\mid nD(r).
$$
Using Lemma~\ref{lem:Dm} (note that $r\ne 6$), we have
\begin{equation}
\label{eq:3}
d(m)   \le   nr \exp\left(-{\sqrt{\log 2 \log r}}\right) = m \exp\left(-{\sqrt{\log 2 \log r}}\right).
\end{equation} 

Using~\eqref{eq:2} for $r \le \sqrt{2} m \exp\left(-{\sqrt{\log 2 \log m + (\log 2)^2/4}}\right)$ and using~\eqref{eq:3} otherwise, 
we complete the proof.

(ii) Let $\varphi_k(n)$ be the $k$-th iterate of the Euler function at $n$. By convention, we set $\varphi_0(n)=n$ and $\varphi_1(n)=\varphi(n)$. For positive integer $n$, define $F(n)$ to be the following square-free integer: 
$$
 F(n) = \prod_{\substack{\textrm{prime $p\mid \varphi_k(n)$ for some $k\ge 1$} \\ p\nmid n}} p.
$$
From~\cite[Theorem~3]{LP2}, there
is a set ${\mathcal T}$ of positive integers  having asymptotic density~1, such that
for $t\to\infty$, $t\in \cT$, we have
$$
  \varphi(t F(t))  \ge t^{(1+o(1))\log\log t/\log\log\log t}.
$$ 
Put $m=tF(t)$, where we remark that $\gcd(t,F(t))=1$. Then  
$$
m \ge t^{(1+o(1))\log\log t/\log\log\log t},
$$
so 
\begin{equation}\label{eq:log0}
\log m\ge (1+o(1))\frac{\log t \log\log t}{\log\log\log t}. 
\end{equation} 
Note that the function $g(x)=x \log\log x / \log x$ is increasing for large $x$. Applying $g$ 
to both sides of~\eqref{eq:log0}, we derive
\begin{equation}\label{eq:log}
\log t\le (1+o(1))\frac{\log m \log\log\log m}{\log\log m}. 
\end{equation} 

So, by~\eqref{eq:log}, we have  
\begin{equation}
\label{eq:t upper}
t\le m^{(1+o(1)) \log\log\log m / \log\log m},
\end{equation}
as $m\to\infty$ through such numbers. 
Besides, noting that $\gcd(t,F(t))=1$ and that $F(t)$ is square-free, 
for each prime factor $p$ of $F(t)$, by the definition of $F(t)$, we obtain 
\begin{align*}
\gcd\left(p,\frac{\delta\lambda(m)\varphi(\rad(m))}{\varphi(m)}\right)
& =\gcd\left(p,\frac{\delta\lambda(m)\varphi(\rad(t))}{\varphi(t)}\right)\\
& = \gcd(p,\delta\lambda(m)) = p. 
\end{align*}
So, we have $F(t) \mid f(m)$, 
which together with~\eqref{eq:t upper} yields 
$$
f(m)\ge F(t) = \frac{m}{t} \ge  m^{1-(1+o(1))\log\log\log m/\log\log m}
$$ 
as $m\to\infty$.
\end{proof}

\section{Proof of Theorem~\ref{thm:fm=dm}} 

Although the result in~\cite[Lemma~2]{LP1} is enough for the proof of Theorem~\ref{thm:fm=dm}, 
we want to take this opportunity to generalize it, which is of independent interest and might have further applications. 

It is shown in~\cite[Lemma~2]{LP1} that there exists a positive constant $c_0$ such that on a set of asymptotic density $1$ of positive integers $m$, $\varphi(m)$  is a multiple of all prime powers $p^a \le c_0 \log\log m/\log\log\log m$. 
The proof of~\cite[Lemma~2.1]{LP3} enhances this result as follows
(note that the particular residue class of prime factors  plays no role in 
this proof):

\begin{lemma}
\label{lem:q1q2}
For  sufficiently large   $x>0$, all but $O(x/\log \log \log x)$ positive integers $m \le x$ have the property that  for any prime power
$p^a \le \log\log x/\log\log\log x$,  
$m$ has  at least two distinct prime factors    congruent to $1$ modulo  ${p^a}$.
\end{lemma} 
Now we are ready to prove Theorem~\ref{thm:fm=dm}. 
This proof follows  some of the arguments in~\cite{LP1}.
\begin{proof}[Proof of Theorem~\ref{thm:fm=dm}]
Let $\cE_1(x)$ be the set of positive integers
$m\le x$ which fail the condition of Lemma~\ref{lem:q1q2}. 
Then we have 
 \begin{equation}
\label{eq:E1}
\# \cE_1(x) = o(x)
\end{equation}
as $x\to\infty$. First we show that for almost all $m$, $d(m)$ does not  have large prime factors. To do this,  for a sufficiently large real number $x$
we put
$$
y=\log\log x,
$$
and let
$$
\cE_2(x)=\{m\le x: p\mid d(m)~{\text{\rm for some odd prime}}~p> y\log y\}.
$$
If $m\in \cE_2(x)$, then $p\mid m$ and $p\mid \varphi(\rad(m))$. So, there is a
 prime factor $q$ of $m$ such that $q\equiv 1\pmod p$.  Hence, $m=pq n$ for some positive integer $n$. The number of such $m\le x$ is
$\lfloor x/pq\rfloor\le x/pq$. Summing this up over all primes $q\le x$ congruent to $1$ modulo $p$ and then over all primes $p>y\log y$ and using the Brun-Titchmarsh theorem  (see~\cite[Theorem~6.6]{IwKow}) coupled with partial summation, we obtain 
 \begin{equation}
 \begin{split}
\label{eq:E2}
\# \cE_2(x) &\le  \sum_{p>y\log y}  \sum_{\substack{q\le x\\ q\equiv 1\pmod p}} \frac{x}{pq}
 =  x\sum_{p> y\log y} \frac{1}{p} \sum_{\substack{q\le x\\ q\equiv 1\pmod p}} \frac{1}{q}\\
& \ll  x\sum_{p \ge y\log y} \frac{\log\log x}{p^2}
\ll x y\sum_{n \ge y\log y} \frac{1}{n^2}
\ll  \frac{x}{\log y} = o(x)
\end{split}
\end{equation}
as $x\to\infty$.  

Let $\cE_3(x)$ be the set of $m\le x$ having a prime divisor in the interval $I=[y/\log y, y\log y]$. Writing $m=pn$, for $p\in I$ and fixing $p$, we get that there are $\lfloor x/p\rfloor\le x/p$ possible choices for $n$. Hence, using the
Mertens formula \eqref{eq:Mert},  
 we obtain 
 \begin{equation}
 \begin{split}
\label{eq:E3}
\#\cE_3(x) & \le  x\sum_{ y/\log y\le p\le y\log y} \frac{1}{p}\\
& =  x\left(\log\log(y\log y)-\log\log( y/\log y)\right) + O(x/\log y)\\
& =  x\log\(1 + O(\log \log y/\log y)\)  + O(x/\log y)\\
& \ll   \frac{x\log\log y}{\log y} = o(x)
\end{split}
\end{equation}
as $x\to\infty$. 

Now, let $\cE_4(x)$ be the set of $m\le x$ which are not in $\cE_1(x)\cup \cE_2(x)\cup \cE_3(x)$ such that $\gcd(m,\varphi(m))$ is divisible by some prime power $p^a> y/\log y$. If $p\mid \varphi(\rad(m))$, since $m$ is not in $\cE_2(x)$, it follows that $p\le y\log y$, and since
 $m$ is not in $\cE_3(x)$, it follows that $p< y/\log y$; hence, $a\ge 2$. If $p \nmid \varphi(\rad(m))$, then we must have $p^2 \mid m$, and since $m \not\in \cE_3(x)$, we have $p< y/\log y$ or $p> y\log y$; so we have $p^2 \mid m$, and either $a\ge 2$ or $p>y \log y$.  
 Thus, $m$ has a square-full divisor $d > y/\log y$ 
 or $d>(y\log y)^2$. Fixing $d$, the number of such  $m\le x$ is $\lfloor x/d\rfloor\le x/d$. 
 So, we deduce that 
 \begin{equation}
\label{eq:E4}
\begin{split}
 \#\cE_4(x) &\le x\sum_{\substack{d>y/\log y\\ d~{\text{\rm square-full}}}} \frac{1}{d}
 +x\sum_{\substack{d>(y\log y)^2\\ d~{\text{\rm square-full}}}} \frac{1}{d} \\
 & \ll \frac{x(\log y)^{1/2}}{y^{1/2}} + \frac{x}{y\log y}  = o(x)
 \end{split}
\end{equation}
as $x\to\infty$. 

We see from~\eqref{eq:E1}, \eqref{eq:E2}, \eqref{eq:E3} and~\eqref{eq:E4} that for the exceptional set 
$$
\cE(x) =  \cE_1(x)\cup \cE_2(x)\cup \cE_3(x)\cup \cE_4(x),
$$
we have 
$$
\# \cE(x) = o(x)
$$
as $x\to\infty$. 
 From now on, we assume that $m\in [1, x]\setminus \cE(x)$, and assume that $m$ has the 
 prime factorisation as in~\eqref{eq:1}. Looking at
 $$
 d(m)=\gcd(m, \delta \varphi(\rad(m))), 
 $$
 we claim that 
 $$
d(m)=\prod_{\substack{1\le j \le k\\ p_j< y/\log y}} p_j^{r_j}.
 $$
 Indeed, since  $m\not\in \cE_2(x)\cup \cE_3(x)$, it follows that if $ p_j \ge y/\log y$, 
 then  $p_j\nmid d(m)$. Further, if $p_j<y/\log y$, then    since $m \not\in \cE_1(x)$, by  Lemma~\ref{lem:q1q2}, we have $p_j \mid \varphi(\rad(m))$, and so from $p_j^{r_j} \mid \gcd(m,\varphi(m))$, we deduce that
 $p_j^{r_j}\le  y/\log y$ because $m\not\in \cE_4(x)$,  and thus $p_j^{r_j}$ divides $\varphi(\rad(m))$ because $m\not\in \cE_1(x)$. This yields the claim. 
 
Finally, we look at 
$$
f(m) = \gcd(m, \delta\lambda(m)\varphi(\rad(m))/\varphi(m)). 
$$
Let $\cE_5(x)$ be the set of positive integers $m\le x$ 
for which there exist a prime $p$ and an integer $r\ge 2$ such that $p^r\mid m$ and $p^r > (y/\log y)^{1/2}$. Thus, each $m\in \cE_5(x)$ has a square-full divisor $d > (y/\log y)^{1/2}$, so as in~\eqref{eq:E4} we deduce that
$$
\# \cE_5(x) = o(x)
$$
as $x\to\infty$. 

Let $m\in [1,x]\setminus (\cE(x) \cup \cE_5(x))$. We still assume that $m$  has the prime factorisation ~\eqref{eq:1}. 
For a prime factor $p_j < y/\log y$, from the above discussion, we have $p_j^{r_j}\le  y/\log y$, 
and then $p_j^{r_j} \mid \lambda(m)$ because $m \not\in \cE_1(x)$, and so $p_j \mid f(m)$. If $r_j\ge 2$, then  $p_j^{r_j}\le ( y/\log y)^{1/2}$ because $m\not\in \cE_5(x)$, and thus $p_j^{2r_j}\le  y/\log y$, which, together with $m\not\in \cE_1(x)$ and Lemma~\ref{lem:q1q2}, implies that there exists a prime factor $q$ of $m$ such that $p_j^{2r_j} \mid q-1$. Thus, $p_j^{r_j} \mid f(m)$. Hence, we have 
 $$
f(m)=\prod_{\substack{1\le j \le k\\ p_j< y/\log y}} p_j^{r_j} = d(m).
 $$
 This completes the proof. 
\end{proof}

\section{Proof of Theorem~\ref{thm: pair df}}  
\label{sec:pair df}

(i) 
Choose an odd prime $p \equiv 1\pmod {n^2}$.  By Linnik's Theorem,  the smallest such prime satisfies 
$$
p\ll n^{2L}.
 $$
 By construction, we have $p>n^2$.

Now, take
$$
m=np.
$$
Clearly, we have 
$$
d(m)=\gcd(m,\delta\varphi(\rad(m)))=\gcd(np, \delta\varphi(\rad(n))(p-1))=n, 
$$
and then noticing $n^2 \mid \lambda(m)$, we have 
\begin{align*}
f(m) & =\gcd(m,\delta\lambda(m)\varphi(\rad(m))/\varphi(m)) \\
& =\gcd(np, \delta\lambda(m)\varphi(\rad(n))/\varphi(n))=n.
\end{align*}

So, we can construct an integer $m \ll n^{2L+1}$ such that  
$$
d(m)=f(m)=n. 
$$

(ii)
We first write $\lambda(n)=\lambda_1\lambda_2$ such that 
$$
\gcd(\lambda_1,n)=1 \mand \rad(\lambda_2) \mid n.
$$
We then choose an odd prime $q$ satisfying 
$$
q \equiv 1 +  8n^3\lambda_2^2/\rad(n) \pmod{8n^4\lambda_2^2/\rad(n)}.
$$ 
By Linnik's Theorem,  the smallest such prime satisfies
$$
q\ll n^{6L}.
 $$
By construction, we can write 
$$
q-1=8kn^3\lambda_2^2/\rad(n)
$$
 with 
$$
\gcd(q,8n^3\lambda_2^2/\rad(n))=1 \mand \gcd(k,n)=1. 
$$

Take
$$
m=4n^2\lambda_2q.
$$
Clearly, we get 
\begin{align*}
d(m)& =\gcd(m,\delta\varphi(\rad(m))) \\
&=\gcd(4n^2\lambda_2q, 2\varphi(\rad(n))(q-1))=4n^2\lambda_2.  
\end{align*}
In addition, note that 
$$
\lambda(m) = \lcm[\lambda(4n^2\lambda_2), q-1] = c(q-1) 
$$
for some integer $c$ dividing $\lambda_1$, and so 
$$
\gcd(c,n)=1.
$$
Then,  
\begin{align*}
\frac{\lambda(m)\varphi(\rad(m))}{\varphi(m)} & =  \frac{c(q-1)\varphi(\rad(n))}{\varphi(4n^2\lambda_2)} \\
& =\left\{ \begin{array}{ll}
      2ckn\lambda_2 & \textrm{if $n$ is even,}\\
      4ckn\lambda_2 & \textrm{if $n$ is odd,}
                 \end{array} \right.
\end{align*}
where we use the identity 
$$
\varphi(4n^2\lambda_2)=\left\{ \begin{array}{ll}
      4n\lambda_2\varphi(n) & \textrm{if $n$ is even,}\\
      2n\lambda_2\varphi(n) & \textrm{if $n$ is odd.}
                 \end{array} \right.
$$
Thus, if $n$ is even, we obtain  
\begin{align*}
f(m)&=\gcd(m,\delta\lambda(m)\varphi(\rad(m))/\varphi(m)) \\ 
& =\gcd(4n^2\lambda_2q,4ckn\lambda_2)=4n\lambda_2;
\end{align*} 
while if $n$ is odd, we get 
$$
f(m) = \gcd(4n^2\lambda_2q,8ckn\lambda_2)=4n\lambda_2.
$$
Hence, we always have $f(m)=4n\lambda_2$, and so 
$$
d(m)/f(m) =n. 
$$ 
We conclude the proof by noticing that we can make $m\ll n^{6L+3}$.

(iii) 
Denote $c=b/a$. If $c$ is even, then since $\gcd(c,\varphi(\rad(b)))=1$, we must have $b=2^r$ for some integer $r\ge 1$. Then, from the assumption $\gcd(c,a)=1$, we see that $a=1$. For $(a,b)=(1,2)$, 
by choosing $m=2^s,s\ge 3$, we get $(f(m),d(m))=(1,2)$. 
If $r\ge 2$, by choosing a prime $p$ such that $2^{r-1}\mid p-1$ and $2^r \nmid p-1$ (for example, $p \equiv 1 + 2^{r-1} \pmod{2^r}$) and putting $m=2^{r+1}p$, we obtain $(f(m),d(m))=(1,2^r)=(a,b)$.

In the following, we assume that $c$ is odd. 
We choose an odd prime $q$ such that
$$
q \equiv 1 +  a^2c \pmod{a^2c^2}.
$$ 
Write $q-1=a^2cj$. By construction, we have $\gcd(c,j)=1$.

Now, let $m=ac^2q$.
Since $\gcd(c,j)=1$ and $\gcd(c,a\varphi(\rad(b)))=1$, it is easy to see that
$$
d(m) = \gcd(ac^2q, \delta\varphi(\rad(b))a^2cj) = ac =b, 
$$
and 
$$
f(m)=\gcd\left(ac^2q,\frac{\delta\lcm[\lambda(a),\lambda(c^2),a^2cj]\varphi(\rad(b))}{\varphi(ac^2)}\right)=a.
$$
 As in the above, by Linnik's Theorem, we can choose
$$
m \ll b^{2L+2}/a.
$$

\section{Proof of Theorem~\ref{thm:except}}
\label{sec:except}

(i)
By contradiction, assume that there exists an integer $m \ge 1$ such that $f(m)=n$ and $d(m)=2n$.
That is, we have
\begin{equation}\label{eq:cond1}
\gcd(m,\delta\varphi(\rad(m))) = 2n, 
\end{equation}
and
\begin{equation}\label{eq:cond2}
 \gcd(m,\delta\lambda(m)\varphi(\rad(m))/\varphi(m)) = n.
\end{equation}
Note that $n>1$ and $n$ is odd. 

Write $m=2nm_1$. Note that we must have $m_1 >1$. If $m_1$ is even, then $\delta=2$, and so $4 \mid \gcd(m,\delta\varphi(\rad(m)))$. Thus,  $4 \mid 2n$ by~\eqref{eq:cond1}, which contradicts  the fact that $n$ is odd.
So, $m_1$ must be odd.

Then it is easy to see that the integer $\lambda(m)\varphi(\rad(m))/\varphi(m)$ is even. 
So, $2 \mid n$ by~\eqref{eq:cond2}. This contradicts the fact that $n$ is odd. Hence, such an integer $m$ does not exist. 

Similarly, by contradiction, we can also show that there is no positive integer $m$ such that $(f(m),d(m))=(n,4n)$.

(ii) 
We choose an odd prime $\ell$ such that
$$
\ell \equiv 1 +  q \pmod{pq}.
$$ 
Write $\ell-1=qa$. By construction, we have $\gcd(p,a)=1$.

Now, let $m=p^2q\ell$.
Since $\gcd(p,a)=1, p\mid q-1$ and $p^2 \nmid q-1$, it is easy to see that
$$
d(m) = \gcd(p^2q\ell, (p-1)(q-1)qa) = pq,
$$
and
$$
f(m)=\gcd(p^2q\ell,\frac{\lcm[p(p-1),q-1,qa]}{p})=q.
$$
 As before, by Linnik's Theorem, we can choose
$$
m \ll p(pq)^{L+1}.
$$

\section{Comments}
\label{sec:comm}

We see from the proof of Theorem~\ref{thm: pair df} that its bounds 
depend on the smallest prime in 
some specific arithmetic progressions and thus in many cases, the value of $L$  can be 
chosen to be smaller than that implied by the general results of  Xylouris~\cite{Xyl1,Xyl2}.

For example, in Theorem~\ref{thm: pair df}~(i) the result depends on the smallest prime 
$p \equiv 1 \pmod n$. We have already mentioned that~\cite{BFI} 
allows the value of $L=2$ for almost all $n$. 
One however can do better with a result of Mikawa~\cite{Mik} that allows to take 
 $L=32/17$ for almost all $n$ in the statement of  Theorem~\ref{thm: pair df}~(i). 
 
 Furrthermore,  in Theorem~\ref{thm: pair df}~(ii) the result depends on the smallest prime 
$q \equiv 1 \pmod {n^2}$.  The result of Baker~\cite{Bak2} (see also~\cite{Bak1}),
which gives a version of the Bombieri--Vinogradov theorem for square moduli, implies 
that for  almost all $n$ the statement of  Theorem~\ref{thm: pair df}~(ii) 
 holds with any fixed $L > 2$. 
 
 We also recall that there are concrete families of moduli which admit a better value of $L$.
 For example, by  the result of Chang~\cite[Corollary~11]{Chang} 
 one can take any $L>12/5$ for moduli without large prime divisors. For example, 
 this is true for all
 powers of a fixed prime number. 
  
 It is also likely that one can  use the above results to improve Theorems~\ref{thm: pair df}~(ii)
 andl~\ref{thm:except}~(ii) for almost all values of the parameters involved.

We remark that the additivity  of the Carmichael quotients  
implies that for any integer $k$ the exponential function $\exp(2 \pi i k C_m(a)/m)$ is a multiplicative 
character of the group  $(\Z/m^2\Z)^*$. For a prime $m=p$,  this has been observed and used
by Heath-Brown~\cite[Theorem~2]{H-B2} (see also~\cite{Shp1}) in the classical case 
of Fermat quotients $(a^{p-1}-1)/p$ modulo a prime $p$.  The same approach 
also works for the Carmichael quotients, and combined with the Burgess 
bound (see~\cite[Theorem~12.6]{IwKow})
allows to study the distribution of  values $C_m(a)$, $1 \le a \le A$, modulo $m$.  One can also study their algebraic  and additive 
properties (see~\cite{ChWi} and~\cite{HarmShp}, respectively, for the case of Fermat quotients). 
Furthermore, using that the set~\eqref{eq:Zero Set} is a subgroup  of $(\Z/m^2\Z)^*$ one can 
obtain analogues of several other results about the distribution of its elements, 
in particular about the smallest element which does not belong to 
this set (see~\cite{BFKS, OstShp, SSV, Shp2,Shp3} and references therein).

\section*{Acknowledgements}
The authors are very grateful to the referee for valuable comments, suggestions and especially for pointing out the error 
in~\cite[Proposition~4.3]{Sha} and for suggesting establishing the results in Theorem \ref{thm: pair df} (i) and (ii). 

This paper started during a visit of F.~L. to the School of Mathematics and Statistics of the University of New South Wales in January 2016.
This author thanks that Institution for hospitality and support. 
Part of this work was also done when F. L. was visiting the Max Planck Institute for Mathematics in Bonn, Germany in 2017. He thanks this institution for hospitality. In addition, he was also supported by grants CPRR160325161141 and an A-rated researcher award both from the NRF of South Africa and by grant no. 17-02804S of the Czech Granting Agency. 
The research of the second and third authors was supported by the Australian
Research Council Grant DP130100237.


\begin{thebibliography}{99}

\bibitem{ADS} 
T. Agoh, K. Dilcher and L. Skula, `Fermat quotients for composite moduli', \textit{J. Number Theory}, \textbf{66} (1997), 29--50. 

\bibitem{Bak1} R.~Baker,
`Primes in arithmetic progressions to spaced moduli',
{\it Acta Arith.\/} {\bf 153} (2012), 133--159.

\bibitem{Bak2} R.~Baker,
`Primes in arithmetic progressions to spaced moduli III',
{\it  Preprint\/}, 2016.  (available at~\url{http://arxiv.org/abs/1602.03500})

\bibitem{BFI}
 E. Bombieri, J. B. Friedlander and H. Iwaniec, `Primes in arithmetic progressions to large moduli. III', \textit{J. Am. Math. Soc.}, \textbf{2} (1989), 215--224.

\bibitem{BFKS} J.~Bourgain, K. Ford, S. V. Konyagin and
I. E. Shparlinski,
`On the divisibility of Fermat quotients',
{\it Michigan Math. J.\/}, {\bf 59} (2010), 313--328.

\bibitem{Chang} M.-C.~Chang,
`Short character sums for composite moduli',
{\it J.\ d'Analyse Math.\/} 2 (2014), 1--33.

\bibitem{ChWi}
Z. X. Chen and A. Winterhof, `Interpolation of Fermat quotients',
{\it SIAM J. Discr. Math.\/}, {\bf 28} (2014), 1--7.


\bibitem{Erdos} P. Erd\H os, `Some asymptotic formulas in number theory', {\it J. Indian Math. Soc.\/}, {\bf 12} (1948), 75--78.
%

\bibitem{ELP} P. Erd\H os, F. Luca and C. Pomerance, `On the proportion of integers coprime to an integer',
in {\it Anatomy of Integers\/} (De Koninck, Granville, Luca, eds.) CRM Proceedings \& Lecture Notes
{\bf 48}, AMS 2008, 47--64.

\bibitem{HarmShp} G. Harman and I. E. Shparlinski,
`Products of small integers in residue classes and additive properties 
of Fermat quotients', 
{\it Intern. Math. Research Notices\/},  {\bf 2016} (2016), 1424--1446.

\bibitem{H-B1} D. R. Heath-Brown, `Zero-free regions for Dirichlet 
$L$-functions, and the least prime
in an arithmetic progression', 
{\it Proc. London Math. Soc.\/}, {\bf 64} (1992), 265--338.

\bibitem{H-B2} R. Heath-Brown, `An estimate for Heilbronn's exponential sum', 
in {\it Analytic Number Theory:
Proc. Conf.  in honor of Heini Halberstam\/}, 
Birkh{\"a}user, Boston, 1996, 451--463. 

\bibitem{IwKow} H. Iwaniec and E. Kowalski,
{\it Analytic number theory\/}, Amer.  Math.  Soc., Providence, RI, 2004.


\bibitem{LP1} F. Luca and C. Pomerance, `On some problems of M\c akowski-Schinzel and Erd\H os concerning the arithmetical functions $\varphi$ and $\sigma$', {\it Colloq. Math. \/} {\bf 92} (2002), 111--130.


\bibitem{LP2} F.~Luca and C.~Pomerance,
`Irreducible radical extensions and Euler-function chains,' {\it
Integers\/} {\bf 7} (2007), Paper A25. 


\bibitem{LP3} F. Luca and C. Pomerance, `The range of the sum-of-proper-divisors function', 
{\it Acta Arith.\/} {\bf 168} (2015),  187--199.

\bibitem{Mik}
H.~Mikawa, `On primes in arithmetic progressions', {\it Tsukuba
J.\ Math.\/} {\bf 25} (2001), 121--153.

\bibitem{OstShp}
A. Ostafe and I.~E.~Shparlinski,  
`Pseudorandomness and dynamics of Fermat quotients',
{\it SIAM J. Discr. Math.\/},   {\bf 25} (2011),  50--71.


\bibitem{Sha} M. Sha, `The arithmetic of Carmichael quotients', {\it Per. Math. Hungarica\/} {\bf 71} (2015), 11--23. (available at \url{http://arxiv.org/abs/1108.2579})

 \bibitem{SSV}  I. D. Shkredov,  E. V. Solodkova and 
 I. V. Vyugin, `Intersections of multiplicative subgroups and 
 Heilbronn's exponential sum', {\it Preprint\/},  2013. 
 (available at \url{http://arxiv.org/abs/1302.3839})

\bibitem{Shp1} I. E. Shparlinski, 
`Fermat quotients: Exponential sums, value set and primitive roots', 
{\it Bull. Lond. Math. Soc.\/}, {\bf 43} (2011), 1228--1238.

\bibitem{Shp2} I. E. Shparlinski, 
`On the value set of Fermat quotients. 
{\it Proc. Amer. Math. Soc.\/}, {\bf 140} (2012),  1199--1206.

\bibitem{Shp3}
I. E. Shparlinski,
`On vanishing Fermat quotients and a bound of the Ihara sum',
{\it Kodai Math. J.\/},  {\bf 36} (2013), 99--108.

\bibitem{Xyl1} 
T. Xylouris, `On the least prime in an arithmetic progression and estimates for the zeros of Dirichlet $L$-functions',
{\it Acta Arith.\/}, {\bf 150} (2011), 65--91.

\bibitem{Xyl2} 
T. Xylouris, `\:Uber die Nullstellen der Dirichletschen L-Funktionen und die kleinste Primzahl in einer arithmetischen Progression', {\it Dissertation zur Erlangung des Doktorgrades, Rheinischen Friedrich- Wilhelms-Universit{\:a}t, Bonn, 2011\/}, Bonner Mathematische Schriften, vol.404. Univ. Bonn, 
Mathem. Institut, Bonn, 2011. 

\end{thebibliography}
\end{document}